\documentclass[12pt]{article}
\usepackage[centertags]{amsmath}
\usepackage{amsfonts}
\usepackage{amssymb}
\usepackage{amsthm}
\input{amssym.def}
\usepackage{graphicx}

\numberwithin{equation}{section}

\newtheorem{Definition}{Definition}[section]
\newtheorem{theorem}[Definition]{Theorem}
\newtheorem{lemma}[Definition]{Lemma}
\newtheorem{proposition}[Definition]{Proposition}
\newtheorem{corollary}[Definition]{Corollary}

\begin{document}
\title{\Large \bf Normal Subgroup  Based Power  Graph of a  finite Group}
\author{A. K. Bhuniya and Sudip Bera}
\date{}
\maketitle

\begin{center}
Department of Mathematics, Visva-Bharati, Santiniketan-731235, India. \\
anjankbhuniya@gmail.com, sudipbera517@gmail.com
\end{center}

\begin{abstract}
For a finite group $G$ with a  normal subgroup $H$, the normal subgroup based power graph of $G$, denoted by $\Gamma_H(G)$ whose vertex set $V(\Gamma_H(G))=(G\setminus H)\bigcup \{e\}$ and two vertices $a$ and $b$ are edge connected  if $aH=b^mH$ or $bH=a^nH$ for some $m, n \in \mathbb{N}$. In this paper we obtain some fundamental characterizations of the normal subgroup based power graph. We show some relation between the  graph $\Gamma_H(G)$ and the power graph $\Gamma(\frac{G}{H})$.  We show that $\Gamma_H(G)$ is complete if and only of $\frac{G}{H}$ is cyclic  group of order $1$ or $p^m$, where $p$ is prime number and $m\in \mathbb{N}$. $\Gamma_H(G)$ is planar if and only if $|H|=2$   or  $3$ and $\frac{G}{H}\cong \mathbb{Z}_2\times \mathbb{Z}_2 \times \cdots \times \mathbb{Z}_2$. Also $\Gamma_H(G)$ is Eulerian if and only if $|G|\equiv |H|$ mod$ 2$.
\end{abstract}

\section{Introduction}
The study of different algebraic structures using graph theory becomes an exciting research topic in the last few decades, leading to many fascinating results and questions, \cite{D},  \cite{atani}, \cite{sen}, \cite{ela},  \cite{r}.  Given an algebraic structure $S$, there are different formulations to associate a directed or undirected graph to $S$, and we can study the algebraic properties of $S$ in terms of properties of associated graphs.

The directed power graph of a semigroup was introduced by Kelarev and Quinn \cite{K}.  If $S$ is a semigroup, then the directed power graph $\overrightarrow{\mathcal{P}(S)}$ of $S$ is a directed graph with $S$ as the set of all vertices and for any two distinct vertices $u$ and $v$ of $S$, there is an arc from $u$ to $v$ if $v=u^m$ for some positive integer $m$. Then in \cite{sen} Chakrabarty, Ghosh and Sen  defined the undirected power graph $\Gamma(S)$ of a semigroup $S$ such that two distinct elements $u$ and $v$ of $S$ are edge connected in $\Gamma(S)$ if $u=v^{m}$ or $v=u^{m}$ for some positive integer $m$. They proved that for a finite group $G$, the undirected power graph $\Gamma(S)$ is complete if and only if $G$ is a cyclic group of order $1$ or $p^{m}$ for some prime $p$ and positive integer $m$. Then Chelvam and Sattanathan proved that the power graph $\Gamma(G)$ is Eulerian if and only if $|G|$ is odd. They also showed that for the finite Abelian  group $G$ the power graph $\Gamma(G)$ of the group $G$ is planar if and only if either $G\cong \mathbb{Z}_2\times \mathbb{Z}_2 \times \cdots \times \mathbb{Z}_2$ or $ \mathbb{Z}_3\times \mathbb{Z}_3 \times \cdots \times \mathbb{Z}_3$ or $ \mathbb{Z}_4\times \mathbb{Z}_4 \times \cdots \times \mathbb{Z}_4$ or $ \mathbb{Z}_2\times \mathbb{Z}_2 \times \cdots \times \mathbb{Z}_2 \times \mathbb{Z}_4\times \mathbb{Z}_4 \times \cdots \times \mathbb{Z}_4$.  In \cite{Cameran}, Cameron and Ghosh showed that for two finite abelian groups $G_1$ and $G_2$, $\mathcal{P}(G_1) \cong \mathcal{P}(G_2)$ implies that $G_1 \cong G_2$. They also conjectured that two finite groups with isomorphic undirected power graphs have the same number of elements of each order. Throughout this paper we denote $\Gamma^*(G)=\Gamma(G)\setminus\{e\}$, where $e$ is the identity of the group $G$ and $a\sim b$ means $a$ is edge connected with $b$.

In this direction  another well studied graph is zero -divisor graph. The zero- divisor graph of a commutative ring  was introduced by I. Beck in \cite{beck}. Let $R$ be a commutative ring with $1$. The zero- divisor graph $\Gamma(R)$ of $R$ is an undirected graph whose vertices are the nonzero zero-divisor of $R$ and two distinct vertices $x, y$ are adjacent if $xy=0$. There are many papers on assigning a graph to a ring, for instance, see \cite{Akbari}, \cite{beck}, \cite{Demeyer}.   An ideal based zero divisor graph of a commutative ring was introduced by Shane P. Redmond in \cite{r}. In \cite{r}, Redmond obtained the basic structure of the ideal based zero divisor graph. He characterized about the connectivity, planarity of the ideal based zero divisor graph. Also he found  the clique number, girth and he showed some  relations between the ideal based zero divisor graph $\Gamma_I(R)$ and the commutative ring $R$. Then some researcher generalized the concept of ideal based zero divisor graph in a semiring. In \cite{atani} Atani investigated the interplay between the semiring theoretic properties of the semiring $R$ and graph theoretic properties of the graph $\Gamma_I(R)$  for some ideal $I$ of $R$. He  showed that two ideals $I$ and $J$ of the semiring $R$ $\Gamma_I(R)=\Gamma_J(R)$ if and only if $I=J$. Let $I$ be an ideal of a semiring $R$. Then if there are nonadjacent elements $a, b\in \Gamma_I(R)$ such that the ideal $J=<a, b>$ is prime to $I$, then diam($\Gamma_I(R))=3$. In \cite{ela} Elavarasan and Porselvi studied some topological properties of an ideal based zero divisor graph of a poset. They also studied the semi ideal based zero divisor graph  structure of a poset $P$, and characterized its diameter.

Motivated by ideal based zero divisor graph we introduce a graph namely normal subgroup based power graph. Let $G$ be a finite group having a normal subgroup $H$. Then the normal subgroup based power graph $\Gamma_H(G)$ is a undirected graph whose vertex set $V(\Gamma_H(G))=(G\setminus H)\bigcup \{e\} $, and two distinct vertices $x$ and $y$ are edge connected if $xH=y^mH$ or $yH=x^nH$ for some $m, n \in \mathbb{N}$. Clearly normal subgroup based power graph $\Gamma_H(G)$ is a generalization of power graph $\Gamma(G)$ in the sense that $\Gamma_{\{e\}}(G)=\Gamma(G)$.

 In Section $2$ we characterize some basic properties of the graph $\Gamma_H(G)$. Then in Section $3, 4, 5$ we give the interplay between the graph theoretic properties of the graph $\Gamma_H(G)$ and the group theoretic properties of the quotient group $\frac{G}{H}$. Finally in Section $6$ we calculate the edge number of the graph $\Gamma_H(G)$ of any finite group $G$ having any normal subgroup $H$, and we find the clique number and the chromatic number of the graph $\Gamma_H(G)$. We show that  for a finite group $G$ having normal subgroup $H$, $\Gamma_H(G)$ is connected, $\Gamma_H(G)$ is neither bipartite nor tree. Also we show that the graph $\Gamma_H(G)$ is   Cayley graph if and only if $\frac{G}{H}$ is a cyclic $P$ group, $\Gamma_H(G)$ is planar if and only if $|H|$ is $2$ or $3$ and $\frac{G}{H}\cong \mathbb{Z}_2\times \mathbb{Z}_2\times \cdots \times \mathbb{Z}_2$. Also the graph $\Gamma_H(G)$ is Eulerian if and only if $|G|\equiv|H|$ (mod) $2$.
\section{Definition and basic structure}
Let $G$ be a finite group having a  normal subgroup $H$. Then the normal subgroup based power graph $\Gamma_H(G)$ is a undirected graph whose vertex set $V(\Gamma_H(G))=(G\setminus H)\bigcup \{e\} $, and two distinct vertices $x$ and $y$ are edge connected if $xH=y^mH$ or $yH=x^nH$ for some $m, n \in \mathbb{N}$.

Now  we give basic structure of the graph $\Gamma_H(G)$ and we show the interplay between the graph $\Gamma_H(G)$ and the power graph $P(\frac{G}{H})$.

\begin{proposition}
Let $G$ be a finite group of order $n$ and $H$ be a  normal subgroup of $G$. Then $\Gamma_H(G)$ is connected.
\end{proposition}
\begin{proof}
Here we show that every vertex of  $\Gamma_H(G)$ is edge connected with $e$, the identity of the group $G$. For any $a\in V(\Gamma_H(G))$, we have $a^n=e$ implying $a^nH=eH$. And so $a$ is edge connected with $e$. Hence $\Gamma_H(G)$ is connected.
\end{proof}
\begin{proposition}
Let $G$ be a finite group of order $n$ and $H$ be a  normal subgroup of $G$. Then all  vertices of the graph $\Gamma_H(G)$ in any coset of $H$ form a clique.
\end{proposition}
\begin{proof}
The vertex set $V(\Gamma_H(G)$ of the graph $\Gamma_H(G)$ can be written as $V(\Gamma_H(G)=\bigcup_{a_iH}\bigcup\{e\}$, where $a_i\in G \setminus H$ Let $a_ih_1, a_ih_2 \in a_iH$, where $ h_1, h_2\in H$. Then $a_ih_1H=a_iH=a_ih_2H$ implying $a_ih_1$ is edge connected with $a_ih_2$. So all vertices in  any coset  of $H$ forms a clique in $\Gamma_H(G)$.
\end{proof}
\begin{corollary}
Let $G$ be a finite group of order $n$ and $H$ be a  normal subgroup of $G$ with $|H|=m$. Then the normal subgroup based power graph $\Gamma_H(G)$ contains atleast $[G:H]-1$ subgraphs isomorphic to $K_m$.
\end{corollary}

\begin{proposition}
Let $a_iH $ and $a_jH$ be two distinct cosets of $H$. If an element of $a_iH$ is edge connected with an element of $a_jH$, then each element of $a_iH$ is edge connected with each element of $a_jH$.
\end{proposition}
\begin{proof}
Suppose $a_ih_i$ is edge connected with $a_jh_j$ for some $h_i, h_j\in H$. Then either $a_ih_iH=(a_jh_jH)^m$ or $a_jh_jH=(a_ih_iH)^n$ for some $m, n\in \mathbb{N}$. Let $a_ih\in a_iH$ and $a_jh_1\in a_jH$. Then $a_ihH=a_iH=(a_jH)^m=(a_jh_1H^m)$ implying $a_ih$ is edge connected with $a_jh_1$. Similarly $a_ih\sim a_jh_1$ if $a_jh_jH=(a_ih_iH)^n$.
\end{proof}
\begin{center}
  \includegraphics[width=12in]{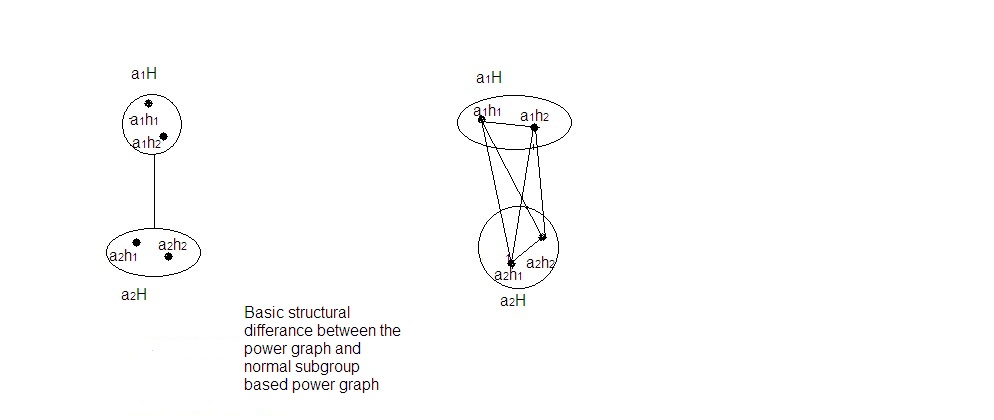}.
 \end{center}
 \begin{corollary}
 Let $|H|=m$. Then $\Gamma_H(G)$ contains atleast $t-[G:H]+1$ subgraphs isomorphic to $K_{2m}$, where $t$ is the number edges in $\Gamma(\frac{G}{H})$.
 \end{corollary}
Let $G$ be  a  finite group and $H$ be a normal subgroup of $G$. Now we give an important relation between  the power graph $\Gamma(\frac{G}{H})$ of the group $\frac{G}{H}$ and the subgroup based power graph $\Gamma_H(G)$.
\begin{proposition}
Let $a$ and $b$ be two distinct vertices of the graph $\Gamma_H(G)$. Then $a$ is edge connected  with $b$ in $\Gamma_H(G)$ if and only if either $aH=bH$ or $aH$ is edge connected with $bH$ in the power graph $\Gamma(\frac{G}{H})$.
\end{proposition}
\begin{proof}
Let $a \sim b$  in $\Gamma_H(G)$. Then either  $aH=b^mH$ or $bH=a^rH$, for some $m, r\in\mathbb{N}$. Clearly in any case $aH$ is edge connected with $bH$ in the power graph $\Gamma(\frac{G}{H})$ of the quotient group $\frac{G}{H}$.

conversely, suppose $aH$ is edge connected with $bH$ in the power graph $\Gamma(\frac{G}{H})$. Then either $aH= (bH)^m=b^mH$ or $bH=(aH)^r=a^rH$, for some $m, r\in\mathbb{N}$. Hence the result.
\end{proof}
 \begin{corollary}
The normal subgroup based power graph $\Gamma_H(G)$ contains atleast $|H|^{[G:H]-1}$ subgraphs isomorphic to $\Gamma(\frac{G}{H})$.
 \end{corollary}
\section{Complete normal subgroup based power graph }
A graph $\Gamma$ is called complete if every pair of distinct  vertices are adjacent. Let $G$ be a group and let $C$ be a subset of $G$ that is closed under taking inverses and does not contain the identity. Then the  Cayley graph $\Gamma(G, C)$ is the graph with the vertex set $G$ and the edge set $E(\Gamma(G, C))=\{gh:hg^-1\in C\}$.
\begin{theorem}
Let $G$ be a finite group of order $n$ and $H$ be a  normal subgroup of $G$. Then the normal subgroup based power graph $\Gamma_H(G)$ is complete if and only if $\frac{G}{H}$ is cyclic  group of order $1$ or $p^m$, for some prime number $p$ and $m\in \mathbb{N}$.
\end{theorem}
\begin{proof}
First suppose that $\frac{G}{H}$ is a cyclic $p-$group, for some prime $p$. Then $|\frac{G}{H}|=p^t$, where $t\in \mathbb{N}$. Consider $a, b \in V(\Gamma_H(G))$. Then the order of the cyclic subgroups $<aH>$ and $<bH>$ of $\frac{G}{H}$ are $p^{k_1}$ and $p^{k_2}$ for some $k_1, k_2 \in \mathbb{N}$.

If $k_1\geq k_2$, then $<bH>\subseteq <aH>$ which implies that there exists $m\in \mathbb{N}$ such that $bH=a^nH$. Thus $a$ and $b$ are adjacent in  $\Gamma_H(G)$. Similar is the case $k_1\leq k_2$.

Conversely suppose that $\Gamma_H(G)$ is complete. Then the power graph $\Gamma(\frac{G}{H})$ is also complete, by Proposition $2.4$. Hence the group $\frac{G}{H}$ is a cyclic $p-$ group, where $p$ is a prime number.
\end{proof}
Now as a consequence, we have the following corollary proved by Chakrabarty et al. in [Theorem 2.12, \cite{sen}].
\begin{corollary}
Let $G$ be a finite group. Then the power graph $\Gamma(G)$ is complete if and only if $G$ is a cyclic group of order $1$ or $p^m$, where $p$ is prime number and $m\in \mathbb{N}$.
\end{corollary}
\begin{proof}
If we consider the normal subgroup $H=\{e\}$. Then the corollary follows from the above theorem.
\end{proof}
\begin{theorem}
Let $G$ be a finite group and $H$ be a  normal subgroup of $G$. Then $\Gamma_H(G)$ is Cayley graph if and only if $\frac{G}{H}$ is cyclic group of order $1$ or $p^m$, where $p$ is prime number and $m\in \mathbb{N}$.
\end{theorem}
\begin{proof}
Let $\frac{G}{H}$  is a cyclic group of order $1$ or $p^m$. Then the normal subgroup based power graph $\Gamma_H(G)$ is complete. Hence a Cayley graph.

Conversely, suppose that the graph $\Gamma_H(G)$ is Cayley graph. Then $\Gamma_H(G)$ is regular. Since the vertex $e$ is adjacent to every other vertices, it follows that $\Gamma_H(G)$ is complete. Hence $\frac{G}{H}$ is cyclic group of order $1$ or $p^m$, where $p$ is prime number and $m\in \mathbb{N}$.
\end{proof}
\begin{corollary}
Let $G$ be a finite group. Then the the power graph $\Gamma(G)$ of the group $G$ is a Cayley graph if and only if $G$ is a cyclic group of order either $1$ or $p^m$, where $p$ is prime number and $m\in \mathbb{N}$.
\end{corollary}
\begin{proof}
We take the normal subgroup $H=\{e\}$. Then the result follows from above theorem.
\end{proof}

\section{Eulerian and Hamiltonian normal subgroup based power graph}
In this section we characterize the finite group $G$ and normal subgroup $H$ such that the the graph $\Gamma_H(G)$ is Eulerian and Hamiltonian.

A graph $\Gamma$ is called Eulerian if it has a closed trail containing all the vertices of $\Gamma$. An useful equivalent characterization of an Eulerian graph is that a graph $\Gamma$ is Eulerian if and only if the vertices of $\Gamma$ is of even degree.
\begin{lemma}
Let $G$ be a finite group. Then for every $v\in G$, degree of $v$ in $\Gamma(G)$ is $\sum \phi(|C_i|)+|v|-1$ and degree of $v$ in $\Gamma^*(G)$ is $\sum \phi(|C_i|)+|v|-2$, where the sum runs over all group $G$ such that $<v>\varsubsetneq C_i$.
\end{lemma}
\begin{proof}
An element $u\in G$ is adjacent to $v$ in $\Gamma(G)$ if and only if $u\in <v>$ or $<v>\varsubsetneq <u>$. Now there are $|v|-1$ numbers of elements $u\in <v>$ which are adjacent to $v$. Since $<v>\varsubsetneq <u>=C$ if and only if $<v>\varsubsetneq <g>$ for every generators $g$ of $C$, it follows that the total number of vertices $u\in G$ adjacent to $v$ such that $<v>\varsubsetneq <u>$ is $\sum_{C_i \in A}\phi(|C_i|)$, where $A=\{ C_i \subseteq G : <v> \varsubsetneq C_i$ is a cyclic subgroup of $G\}$.
\end{proof}
\begin{theorem}
Let $G$ be a finite group and $H$ be normal subgroup of $G$. Then the graph $\Gamma_H(G)$ is Eulerian if and only if $|G|\equiv|H|$(mod)$2$.
\end{theorem}
Suppose that $|G|\equiv|H|$(mod)$2$. We show that degree of each vertex of the graph $\Gamma_H(G)$ is even. Let $a\in G\setminus H$. Considering $aH$ as a  vertex in $\Gamma(\frac{G}{H})$, by Lemma $4.1$ we have degree of $aH$ is $d(aH)=\sum^n_{i=1}\phi(|(a_iH)|)+|(aH)|-1$. so there are $\sum\phi(|(a_iH)|)+|(a_iH)|-2$ cosets $bH\neq H$ adjacent to $aH$ in $\Gamma(\frac{G}{H})$. Also $aH\sim bH$ in $\Gamma\frac{G}{H}$ implies that $a\sim bh$ for all $h\in H$ and hence these $\sum\phi(|(a_iH)|)+|(aH)|-2$ non-identity adjacent cosets of $H$ contributes $|H|(\sum\phi(|(a_iH)|)+|(aH)|-2)$ to the degree of $a$ in $\Gamma_H(G)$. Since $V(\Gamma_H(G)=G\setminus H\bigcup\{e\}$, the fact $eH\sim aH$ in $\Gamma(\frac{G}{H})$ contributes only $1$ to the degree of $a$ in $\Gamma_H(G)$ for $a\sim e$. Thus we have the degree of $a$  in $\Gamma_H(G)$ is
\begin{center}
$|H|(\sum^n_{i=1}\phi(|(a_iH)|)+|(aH)|-2)+1+(|H|-1)$
=$|H|(\sum^n_{i=1}\phi(|(a_iH)|)+|(aH)|-1)$.
\end{center}
Suppose $|H|$ is even. Then $d(a)$ is even. let $|H|$ is odd. Then $|G|\equiv|H|$(mod)$2$ implies that $|G|$ is odd. Then $|\frac{G}{H}|$ is odd, this implies that $\phi(|a_iH|)$ is even, for all $a_iH \in \frac{G}{H}$ expect $H$ the identity of $\frac{G}{H}$. So $d(a)$ is even. Again $d(e)=|G|-|H|$ is even by the given condition. Hence the degree of each vertex of the graph $\Gamma_H(G)$ is even.
Conversely let the graph $\Gamma_H(G)$ is Eulerian. Then the degree of $e=|G|-|H|$ is even. Hence $|G|\equiv|H|$(mod)$2$.

\begin{center}
  \includegraphics[width=12in]{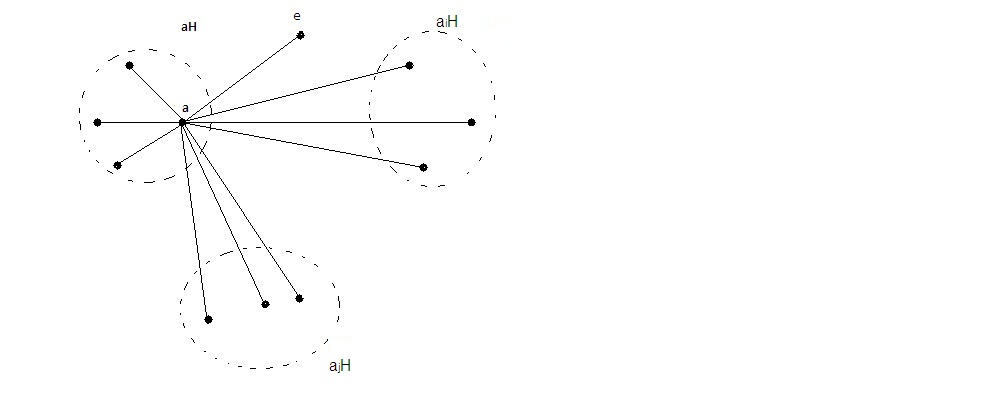}.
 \end{center}
 As an immediate consequence, we have the following characterization of the Eulerian power graphs. This result was proved by Chelvam and Sattanathan. [Theorem 6,\cite{Chelvam} ].
 \begin{corollary}
 The power graph $\Gamma(G)$ of the group $G$ is Eulerian if and only if $|G|$ is odd.
 \end{corollary}
 A graph $\Gamma$ is called Hamiltonian if it has  a cycle that meets every vertex. Such a cycle is called a Hamiltonian cycle.
\begin{theorem}
Let $G$ be a finite group and $H$ be a normal subgroup.  Then $\Gamma_H(G)$ is Hamiltonian  if  the power graph $\Gamma(\frac{G}{H})$ is Hamiltonian.
\end{theorem}
\begin{proof}
Let the power graph $\Gamma(\frac{G}{H})$ is Hamiltonian. Suppose $H=\{ h_1, h_2, \cdots, h_k\}$. Since $\Gamma(\frac{G}{H})$ is Hamiltonian, there exists a Hamiltonian cycle $T=H\sim a_1H\sim a_2H\sim \cdots \sim a_nH\sim H$. Since $a_iH \sim a_jH$ in $\Gamma(\frac{G}{H})$ implies that $a_ih_l \sim a_jh_m$ in $\Gamma_H(G)$ for all $l, m=1, 2, \cdots, k$ and every coset $a_iH$ formes a clique in $\Gamma_H(G)$, so we can construct a Hamiltonian cycle in $\Gamma_H(G)$ as follows:
$$ e\sim a_1h_1 \sim a_1h_2 \sim \cdots \sim a_1h_k \sim a_2h_1 \sim a_2h_2 \sim \cdots \sim a_2h_k \sim \cdots \sim a_nh_1\sim a_nh_2 \sim \cdots \sim a_nh_k \sim e$$.
\end{proof}
In \cite{sen}, Chakrabarty, Ghosh and Sen proved that the power graph $\Gamma(G)$ is Hamiltonian  for every finite cyclic group $G$ of order $\geq 3$. So by the above theorem it follows immediately that:
\begin{corollary}
Let $G$ be a finite cyclic group and $H$ be a normal subgroup such that $[G:H]\geq 3$. Then the graph $\Gamma_H(G)$ is Hamiltonian.
\end{corollary}
\begin{center}
  \includegraphics[width=12in]{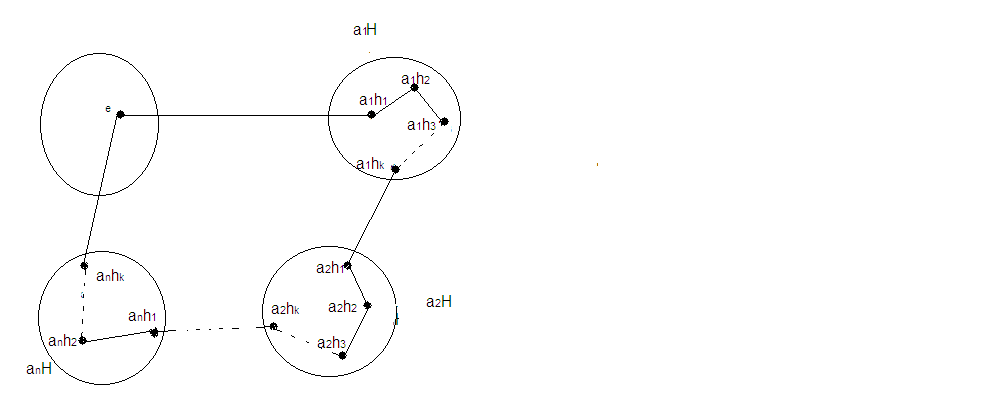}.
 \end{center}
\section{Planar normal subgroup based power graph }
In this section, we characterize the planarity of the normal subgroup based power graphs.

First we show that a normal subgroup based power graph can neither be a tree nor a bipartite graph. Recall that a graph $\Gamma$ is called a tree if it has no cycle. A bipartite graph $\Gamma$ such that the vertex set $V(\Gamma)$ of $\Gamma$ is a disjoin union $V_1\bigcup V_2$, where no two vertices in $V_i$ are adjacent, for $i=1, 2$.
\begin{theorem}
Let $G$ be a finite group  and $H$ be a nontrivial proper normal subgroup.  Then the normal subgroup based  power graph $\Gamma_H(G)$ contains atleast one $3-$ cycle.
\end{theorem}
\begin{proof}
Consider $a\in G\setminus H$. Then $aH\neq H$. Since $H$ is nontrivial, so $|aH|\geq 2$. Let $b, c \in aH$. Then $b$ and $c$ are adjacent and hence $b, c, e$ forms a $3-$cycle in $\Gamma_H(G)$.

Note that though, in the above theorem , we have not assumed  $|G|\geq 3$, yet it follows, otherwise $G$ would have no  nontrivial proper normal subgroups.

Also $H$ is considered to be a nontrivial subgroup, just to avoid the case:  every element of $G$ is of order $2$ and $H=\{e\}$. Chelvam et al. \cite{Chelvam} proved that in this case $\Gamma_H(G)=\Gamma(H)$ is a tree. Even if every element of $G$ is its own inverse, then also $\Gamma_H(G)$ contains a $3-$ cycle for nontrivial proper subgroup $H$ of $G$.
\end{proof}
Since a bipartite graph can not have an odd cycle, we have:
\begin{corollary}
If $G$ is a finite group with a nontrivial  proper normal subgroup $H$, then the normal subgroup based power graph  $\Gamma_H(G)$ is neither bipartite nor tree.
\end{corollary}

The girth of a graph $\Gamma$, denoted by gr$(\Gamma)$, is the length of the shortest cycle in $\Gamma$ and  gr$(\Gamma)$ is infinite if $\Gamma$ has no cycles. Hence we have:
\begin{corollary}
Let $G$ be a  finite group with a nontrivial proper normal subgroup $H$. Then the girth of the normal subgroup based power graph $\Gamma_H(G)$ is $3$.
\end{corollary}
  A graph $\Gamma$ is called  planer if it can be drawn in a plane so that  no two edges intersect. A graph is planer if and only if it does not contain a graph which is isomorphic to either of the graphs $K_{3, 3}$ and $K_5$. In the following we characterize planar normal subgroup based power graphs. Here $|aH|$ denote the number of elements in the coset $aH$ and $o(aH)$ means the order of the element $aH$ in the quotient group $\frac{G}{H}$.
\begin{theorem}
Let $G$ be a finite group and $H$ be a nontrivial proper normal subgroup of $G$. Then the normal subgroup based power graph  $\Gamma_H(G)$ is planar if and only if $|H|=2$ or $3$ and $\frac{G}{H}\cong \mathbb{Z}_2\times \mathbb{Z}_2 \times \cdots \times \mathbb{Z}_2$.
\end{theorem}
\begin{proof}
Suppose that the graph $\Gamma_H(G)$ is planar. Let $|H|>3$. Consider  $a\in G\setminus H$.  Then $|aH|\geq 4$.  Now four vertices of the graph $\Gamma_H(G)$ in the coset $aH$ and the identity $e$ of the group $G$ forms a subgraph which is isomorphic to  $K_5$. So the graph $\Gamma_H(G)$ is not planar. Thus $|H|=2$ or $3$.

Then the proof of the necessary part will be completed if we can show order of every nonidentity  element in the quotient group is $2$. If possible on the contrary, assume that $aH \in \frac{G}{H}$ is such that $o(aH)\geq 2$. Then $a^2H \neq H$ implies that $a^2$ does not belongs to $H$ and $a^2$ is adjacent with $a$ in $\Gamma_H(G)$. Also for $h_1, h_2 \in H $, $ah_1, ah_2, a^2h_1, a^2h_2, e$ forms a subgraph of $\Gamma_H(G)$ isomorphic to $K_5$, by Proposition $2.4$. This contradicts that $\Gamma_H(G)$ is planar. Hence order of every non-identity element of $\frac{G}{H}$ is $2$.

Conversely,  $\frac{G}{H}\cong \mathbb{Z}_2\times \mathbb{Z}_2 \times \cdots \times \mathbb{Z}_2$ implies that order of every non-identity element of $\frac{G}{H}$ is $2$.  Hence if for $a, b \in G \setminus H$, $aH$ and $bH$ are non-identity elements of $\frac{G}{H}$ then $a$ and $b$ are not adjacent in $\Gamma_H(G)$. Also note that the elements of $G\setminus H$ in a coset $aH$ formes a clique. Since $|aH|=2$ or $3$ for every $a\in G\setminus H$, it follows that $\Gamma_H(G)$ has no subgraph isomorphic to $K_5$ or $K_{3, 3}$.
\end{proof}
\begin{center}
  \includegraphics[width=12in]{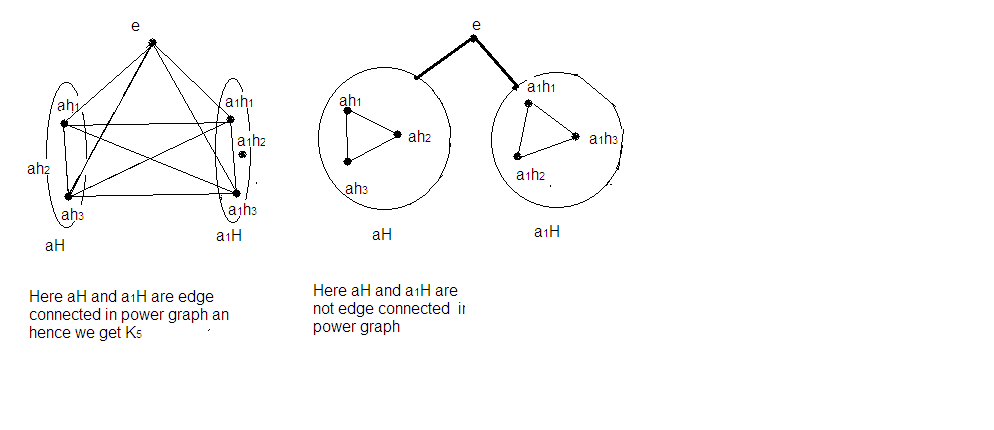}.
 \end{center}

\section{Some parameters of normal subgroup based power graph }
In this section we give the edge number and girth of the normal subgraph based power graph $\Gamma_H(G)$ of any finite group $G$ having a nontrivial normal subgroup $H$.  Also  the clique number and the chromatic number of the normal subgraph based power graph $\Gamma_H(G)$ have been found for $G$ is a finite cyclic group.  In \cite{sen} Chakrabarty et. al proved that if $G$ is a finite group of $n$ elements then total number of edges in $\Gamma(G)$ is $\frac{1}{2}\sum_{a\in G}(2o(a)-\phi(o(a))-1)$. Thus if $H$ is a normal subgroup of $G$ and $[G:H]=n$. Then the total number of edges in the power graph $\Gamma_H(G)$ is $\frac{1}{2}\sum_{aH\in \frac{G}{H}}(2|aH|-\phi(|aH|)-1)$.

\begin{theorem}
Let $G$ be a finite group and $H$ be a nontrivial normal subgroup of $G$ such that $|\frac{G}{H}|=n$. Then the number of edges of the graph $\Gamma_H(G)$ is $(t-n+1)|H|^2+\left(
                                 \begin{array}{c}
                                   |H| \\
                                   2   \\
                                 \end{array}
                               \right)(\frac{|G|}{|H|}-1)+(|G|-|H|)$, where $t=\frac{1}{2}\sum_{aH\in \frac{G}{H}}(2o(aH)-\phi(o(aH))-1)$.

\end{theorem}
\begin{proof}
Let $a_1H, a_2H, \cdots, a_{n-1}H, H $ be the all distinct coset of $H$ in $G$. Then $a_1H \bigcup a_2H \bigcup\cdots \bigcup a_{n-1}H \bigcup\{e\}$ is the  set of all vertices of $\Gamma_H(G)$.  If $a_iH$ and $a_jH$  are adjacent in the power graph $\Gamma(\frac{G}{H})$, then every element of $a_iH$ is adjacent with every element of $a_jH$ in the normal  subgroup based power graph $\Gamma_H(G)$. Thus each edge in $\Gamma(\frac{G}{H})$ between any two vertices of $a_1H, a_2H, \cdots, a_{n-1}H$ induces $|H|^2$ edges in $\Gamma_H(G)$. Suppose there are $t$ edges in the power  graph $\Gamma_H(G)$.  The identity element $H$ of $\frac{G}{H}$ is adjacent to each of the $n-1$ non-identity elements  $a_1H, a_2H, \cdots, a_{n-1}H$ in the power graph  $\Gamma(\frac{G}{H})$. Now by Proposition $2.2$, the vertices of $\Gamma_H(G)$ in each $a_iH$ form a clique and hence there are $\left(
                                                                  \begin{array}{c}
                                                                    |H| \\
                                                                    2 \\
                                                                  \end{array}
                                                                \right)(n-1)
$ edges joining two vertices lying in a coset. Furthermore there are $|G|-|H|$ edges incident with $e$ in $\Gamma_H(G)$. Thus total number of edges in $\Gamma_H(G)$ is   $(t-n+1)|H|^2+\left(
                                 \begin{array}{c}
                                   |H| \\
                                   2   \\
                                 \end{array}
                               \right)(\frac{|G|}{|H|}-1)+(|G|-|H|)$.

\end{proof}
Now by Corollary $4.3$ of \cite {sen}, the following corollary is immediate.
\begin{corollary}
Let $G$ be a finite cyclic group having a  nontrivial subgroup $H$ such that $\frac{|G|}{|H|}=n$. Then the number of edges of the graph $\Gamma_H(G)$ is

$[\frac{1}{2}\sum_{d|n}\{2d-\phi(n)-1\}\phi(d)-n+1]|H|^2+\left(
                                                          \begin{array}{c}
                                                            |H| \\
                                                            2 \\
                                                          \end{array}
                                                        \right)(\frac{|G|}{|H|}-1)+(|G|-|H|).
$
\end{corollary}
Also,  by corollary $4.4$ of \cite{sen}, we have the following corollary.
\begin{corollary}
Let $G$ be a finite abelian group and $H$ be its any nontrivial subgroup such that $\frac{G}{H}\cong G_1\times G_2 \times \cdots \times G_r$. Then the number of edges of the graph $\Gamma_H(G)$ is $(\frac{1}{2}\sum_{d|n_i, i=i, 2, \cdots r}[2lcm\{d_1d_2\cdots d_r\}-\phi(lcm\{d_1d_2\cdots d_r\})-1]-n+1)|H|^2+\left(
                                                  \begin{array}{c}
                                                    |H| \\
                                                    2  \\
                                                  \end{array}
                                                \right)(\frac{|G|}{|H|}-1)+(|G|-|H|)
$, where each $G_i$ is a cyclic group of order $n_i$, $i=1, 2, \cdots,  r$.
\end{corollary}

Let $\Gamma$ be a graph. Then the maximum size of a clique in $\Gamma$ is called the clique number of $\Gamma$  and is denoted by $w(\Gamma)$. In $\Gamma_H(G)$, $a\sim b$ if and only if $ah_1\sim bh_2$ for every $h_1, h_2 \in H$, and $e\sim a$ for every $a\in G\setminus H$, so it follows that clique number of $\Gamma_H(G)$ is of the form $|H|q+1$.
\begin{theorem}
Let $G$ be a finite group and $H$ be a normal subgroup. Then $w(\Gamma_H(G))=|H|(M-1)+1$, where $M=w(\Gamma(\frac{G}{H}))$.
\end{theorem}
\begin{proof}
Let $M=w(\Gamma(\frac{G}{H}))$ and let $\{ H, a_1H, a_2H, \cdots, a_{M-1}H\}$ be a clique in $\Gamma(\frac{G}{H})$ with the maximum size $M$. Suppose that $H=\{ e, h_1, h_2, \cdots, h_n  \}$. Then $\{ e, a_1, a_1h_1,\cdots, a_1h_n,a_2, \cdots, a_2h_n, \cdots, a_{M-1}, a_{M-1}h_1, a_{M-1}h_n \}$ is a clique in $\Gamma_H(G)$, by Proposition $2.2$, Proposition $2.4$, Proposition $2.6$. So $w(\Gamma_H(G))\geq |H|(M-1)+1$. Let $C$ be clique in $\Gamma_H(G)$ of maximum size. Then by Proposition $2.4$, $a\in C$ implies that $aH\subset C $ and hence $w(\Gamma_H(G))=|H|q+1$ where $q$ is then number of distinct coset in $C$. Now by Proposition $2.6$ these $q$ distinct cosets together with $H$ form a clique in $\Gamma(\frac{G}{H})$ and so $q+1\leq M$. Thus $ w(\Gamma_H(G))\leq |H|(M-1)+1$. Hence $w(\Gamma_H(G))=|H|(M-1)+1$.
\end{proof}

The chromatic number $\chi(\Gamma)$ of the graph $\Gamma$ is the smallest positive integer $r$ such that the vertices of $\Gamma$ can be coloured in $r$ colours so that no two adjacent vertices share the same colour.
A graph  $\Gamma$ is called perfect if $w(\Gamma)=\chi(\Gamma)$. Chudnovsky et. al. proved that a graph $\Gamma$ is perfect if and only if neither $\Gamma$ nor its complement   $\overline{\Gamma}$ contains an odd cycle of length atleast $5$ as an induced subgraph.
Here we show that every normal subgroup based power graph is perfect.
\begin{proposition}
Let $H$ be a normal subgroup of a finite group $G$. Then  $\Gamma_H(G)$ is perfect.
\end{proposition}
\begin{proof}
For the sake of simplicity, we prove that  neither $\Gamma_H(G)$ nor $\overline{\Gamma_H(G)}$ contains any 5- cycle  as  an induced subgraph. The general case is similar.  Let if possible $a_1 a_2a_3a_4a_5$ be a five cycle. Suppose $a_2H=a_1^{n_1}H$. If $a_3H=a_2^{n_2}H, a_3H=a_1^{n_1n_2}H$ implies that $a_1\sim a_3$a contradiction. So $a_2H=a_3^{n_3}H$. Similarly if $a_3H=a_4^{n_4}H$ we have a contradiction. Thus $a_4H=a_3^{n_5}H$. Arguing similarly we have $a_4H=a_5^{n_6}H$. Now if $a_1H=a_5^{n_7}H$ then $a_5\sim a_2$ and if $a_5H=a_1^{n_8}H$, $a_1\sim a_4$, a contradiction.

Now we show that $\overline{\Gamma_H(G)} $ does not contain any $5$-cycle as an induced subgraph. In fact, if $b_1 b_2 b_3 b_4 b_5$ is an induced $5$-cycle in $\overline{\Gamma_H(G)} $, $b_1 b_3 b_5 b_2 b_4$ becomes automatically an induced $5$-cycle in $\Gamma_H(G)$.
\end{proof}
Combining the above proposition and the theorem $6.4$ we get the following corollary:
\begin{corollary}
 $\chi(\Gamma_H(G))=|H|(M-1)+1=w(\Gamma_H(G))$.
\end{corollary}

The vertex connectivity  of a graph $\Gamma$ is denoted by $\kappa(\Gamma)$, is the minimum number of vertices whose deletion increases the number of connected component of the graph $\Gamma$ or has only one vertex. Since the vertices of $\Gamma_H(G)$ in every coset form a clique, to increase the number of connected component, we have to delete packets of $|H|$ vertices. Also $e$ is connected to every vertices. Hence $\kappa(\Gamma_H(G))$ is of the form $|H|t+1$.
\begin{theorem}
Let $G$ be a finite group and $H$ be a normal subgroup of $G$.  Then the vertex connectivity of the normal subgroup based power graph $\Gamma_H(G)$ is $(k-1)|H|+1$, where $k=\kappa(\Gamma(\frac{G}{H}))$.
\end{theorem}
\begin{proof}
Let $\kappa(\Gamma(\frac{G}{H}))=k$ and we delete the vertices $H, a_1H, \cdots, a_{k-1}H$ to increase the number of connected components of $\Gamma(\frac{G}{H})$. Since the vertices of $\Gamma_H(G)$ in every coset form a clique, to reflect the deletion of $\Gamma(\frac{G}{H})$ on $\Gamma_H(G)$, we have to delete all of $|H|(k-1)$ vertices in $a_1H\bigcup a_2H \bigcup \cdots \bigcup a_{k-1}H$ and identity $e$. Then by Proposition $2.6$, $\kappa(\Gamma_H(G))\leq |H|(k-1)+1$. Also it follows from Proposition $2.6$, that  the deletion of  less number of vertices than $|H|(k-1)+1$ in $\Gamma_H(G)$ increases the number of connected components,  contradicting that $\kappa(\Gamma(
 \frac{G}{H}))=k$. Thus $\kappa(\Gamma_H(G))=|H|(k-1)+1$.
\end{proof}

\bibliographystyle{amsplain}

\begin{thebibliography}{10}
\baselineskip 5mm


\bibitem{survey}
 Abawajy, J.,  Kelarev, A. V.,  Chowdhury, M. (2013).  Power graphs: A survey. \emph{Electron. J. Graph Theory Appl}. 1:125-147.

\bibitem{anserson}
 Anderson, D. D., Naseer, M. (1993).  Beck's coloring of a commutative ring.
\emph{J. Algebra}. 159:500-514.

\bibitem{D}
 Anderson, D. F.,  Livingston, P. S. (1999).  The zero-divisor graph of a commutative ring.
\emph{J. Algebra}. 217:434-447.

\bibitem{atani}
 Atani, S. E. (2009).  A ideal based zero divisor graph of a commutative semiring. \emph{glasnik matematicki}. 44(64):141-153.

\bibitem{Akbari}
 Akbari, S.,  Maimani, H. R.,  Yassemi, S. (2003).  When a zero divisor graph is planar or a  complete r-partite graph.\emph{J. Algebra}. 270:169-180.

\bibitem{beck}
 Beck, I. (1988).  coloring of commutative ring.\emph{J. Algebra}. 116:208-226.


\bibitem{Cameran}
  Cameron, P. J.,  Ghosh, S. (2011).  The power graph of a finite group. \emph{Discrete Math}.311:220-1222.

\bibitem{Cameran 2}
 Cameron, P. J. (2010).  The power graph of a finite group, II. \emph{J.Group Theory}. 13(6):779-783.

 \bibitem{dostabadi}
 Doostabadi,  A.,   Erfanian, A.,    Jafarzadeh,  A.  (2013).  Some results on the power graph of groups. \emph{The 44 th Annual Irnian Mathematics Conferance}. 27-30, Ferdowsi University of Mashhad, Iran.

 \bibitem{Chelvam}
Tamizh Chelvam,  T.,  Sattanathan,  M. (2013).  Power graph of finite abelian groups. \emph{Algebra and Discrete Mathematics}. 16(1):33-41.



 \bibitem{sen}
 Chakrabarty, I.,  Ghosh, S.,  Sen, M. K. (2009).   Undirected power graphs of semigroups, \emph{Semigroup
Forum }.                                                                                                                78:410-426.

\bibitem{Demeyer}
 DeMeyer, F.,   DeMeyer, L. (2005).  Zero divisor graph of a semigroups, \emph{Journal of algebra}.283:190-198.


\bibitem{ela}
 Elavarasan, B.,  Porselvi,  K. (2013).  An ideal based zero divisor graph of posets, \emph{Commun. Korean Math. Soc.} 28:79-85.

 \bibitem{Godsil}
Godsil,  Chris .,   Royle, Gordon. (2001).  Algebraic Graph Theory, Springer-Verlag, New York Inc.

\bibitem{Hunger}
 Hungerford, T. W. (1974).   Algebra, Gratuets Text in Mathematics, New
York(NY), Springer-Verlag, 73.

\bibitem{K}
 Kelarev, A. V.,   Quinn, S. J. (2002).  Directed graph and combinatorial properties of semigroups. \emph{J. aigebra}, 251:16-26.

\bibitem{r}
 Redmond, S. P. (2003).  An ideal-based zero divisor graph of acommutative ring, \emph{Communication in algebra}. 31:4425-4443.

\bibitem{sing}
Singh, G.,  Manilal, K. (2010).  Some Generalities on Power Graphs  and Strong Power
Graphs, Int. J. Contemp. Math Sciences 5(55):2723-2730.

\bibitem{w}
 West, D. B. (2001). Introduction to Graph theory, 2nd ed. pearson education.



\end{thebibliography}

\end{document}